\documentclass[11pt,leqno]{article}
\usepackage{graphicx}
\usepackage{amsfonts, amsthm}
\usepackage{amsxtra}
\usepackage{tikz}
\usepackage[latin1]{inputenc}
\usepackage{fontenc}
\usepackage[dvips]{epsfig}
\usepackage{amsmath,amsfonts,amssymb}

\begin{document}
\newtheorem{thm}{Theorem}[section]
\newtheorem{cor}[thm]{Corollary}
\newtheorem{lem}[thm]{Lemma}
\newtheorem{exm}[thm]{Example}
\newtheorem{prop}[thm]{Proposition}
\theoremstyle{definition}
\newtheorem{defn}[thm]{Definition}
\newtheorem{rem}[thm]{Remark}
\numberwithin{equation}{section}
\begin{center}
{\LARGE\bf Notes on  Krasnoselskii-type fixed-point theorems and their application
to fractional hybrid differential problems}
\\
\vspace{.25in} {\large{H. Akhadkulov\footnote{School of Quantitative Sciences, University Utara Malaysia, CAS 06010, UUM  Sintok, Kedah Darul Aman, Malaysia. Email: habibulla@uum.edu.my (H.A), yuanying@uum.edu.my (T.Y.Y), azizan.s@uum.edu.my (A.B.S), linda@uum.edu.my (H.I)}},
T. Y. Ying$^2$, A. B. Saaban$^2$,
M. S. Noorani\footnote{School of Mathematical Sciences Faculty of Science and Technology University Kebangsaan Malaysia, 43600 UKM Bangi, Selangor Darul Ehsan, Malaysia.
Email: msn@ukm.edu.my}, and H. Ibrahim$^2$}
\end{center}

\title{Notes on  Krasnoselskii-type fixed-point theorems and their application
to fractional hybrid differential problems}

\begin{abstract}
 In this paper we prove
a new version of Kransoselskii's fixed-point
theorem under a ($\psi, \theta, \varphi$)-weak contraction
condition. The theoretical result is applied to
 prove the existence of a solution of the following fractional hybrid differential equation involving
 the Riemann-Liouville differential and integral operators  orders of $0<\alpha<1$ and $\beta>0:$
 \begin{equation}\nonumber
 \left\{\begin{array}{ll}
D^{\alpha}[x(t)-f(t, x(t))]=g(t, x(t), I^{\beta}(x(t))), \,\,\, \text{a.e.}
\,\,\, t\in J,\,\, \beta>0,\\
x(t_{0})=x_{0},
\end{array}
\right.
 \end{equation}
 where $D^{\alpha}$ is the Riemann-Liouville fractional derivative  order of $\alpha,$
 $I^{\beta}$ is Riemann-Liouville fractional integral operator
order of $\beta>0,$  $J=[t_{0}, t_{0}+a],$ for some fixed $t_{0}\in \mathbb{R},$ $a>0$
 and the functions $f:J\times \mathbb{R}\rightarrow \mathbb{R}$ and
 $g:J\times \mathbb{R}\times \mathbb{R}\rightarrow \mathbb{R}$ satisfy certain conditions.
 An example is also furnished to illustrate the hypotheses and the
abstract result of this paper.
\end{abstract}

\section{Introduction}
Fixed-point theory has experienced quick improvement over
the most recent quite a few years. The development has
been firmly advanced by the vast number of utilizations in
the existence theory of functional, fractional, differential,
partial differential, and integral equations.
Two fundamental theorems concerning fixed points are those of Schauder and
of Banach.
The Schauder's fixed point theorem, involving
a compactness condition, may be stated as "\emph{if $S$ is a closed convex and
bounded subset of a Banach space $X,$ then every completely continuous operator $A:S\rightarrow S$ has at
least one fixed point}".
Note that an operator $A$ on a Banach space $X$ is called completely continuous
if it is continuous and $A(D)$ is totally bounded for any bounded subset $D$ of $X.$
Banach's fixed point theorem, involving a metric assumption on the mapping, 
states that "\emph{if $X$ is complete metric space and if $A$ is a contraction on $X,$ then it has a unique fixed
point, i.e., there is a unique point $x^{*}\in X$ such that $Ax^{*}=x^{*}.$ Moreover, the sequence $A^{n}x$
converges to $x^{*}$  for every $x\in X,$}".
The idea of the hybrid fixed point theorems, that is,  a blend of the nonlinear contraction principle and Schauder's fixed-point theorem
goes back to 1964, with Krasnoselskii \cite{Krasno}, who  still maintains an interest in the subject.
He gave intriguing applications to differential equations by finding the existence
of solutions under some hybrid conditions.
Burton \cite{Bur} extended Krasnoselskii's result for a wide class of operators in 1998.
In 2013, Dhage \cite{Dhag1} and Dhage and Lakshmikantham \cite{Dhag2} proposed an important Krasnoselskii-type
fixed-point theorems and applied them the following first-order hybrid differential equation with linear perturbations
of first type:
 \begin{equation}\label{0IE1}
 \left\{\begin{array}{ll}
\frac{d}{dx}\Big[\frac{x(t)}{f(t,x(t))}\Big]=g(t, x(t)), \,\,\, \text{a.e.}
\,\,\, t\in J,\\
x(t_{0})=x_{0}\in \mathbb{R},
\end{array}
\right.
 \end{equation}
 where $J=[t_{0}, t_{0}+a],$ for some fixed $t_{0}\in \mathbb{R},$ $a>0$
 and $f\in C(J\times \mathbb{R}, \mathbb{R}\setminus \{0\}),$
 $g\in \mathcal{C}(J\times \mathbb{R}, \mathbb{R}).$
In the same year, Dhage and Jadhav \cite{Dhag3} studied the existence of
solution for hybrid differential equation with linear perturbations of second type:
\begin{equation}\label{1IE1}
 \left\{\begin{array}{ll}
\frac{d}{dx}\Big[x(t)-f(t,x(t))\Big]=g(t, x(t)),
\,\,\, t\in J,\\
x(t_{0})=x_{0}\in \mathbb{R},
\end{array}
\right.
 \end{equation}
 where $f,g\in C(J\times \mathbb{R}, \mathbb{R}\setminus \{0\}).$
 They established the existence and uniqueness results and some fundamental differential
 inequalities for hybrid differential equations initiating the study of theory of such systems and proved
 utilizing the theory of inequalities, its existence of extremal solutions and a comparison result.
In \cite{Lu},  Lu et al. proved an existence theorem  for
fractional hybrid differential equations under a $\varphi$-Lipschitz contraction condition
and applying this theorem they develop the theory of fractional hybrid differential equations with
linear perturbations of second type involving Riemann-Liouville differential operators
 order of $0<q<1:$
\begin{equation}\label{2IE1}
 \left\{\begin{array}{ll}
D^{q}\Big[x(t)-f(t,x(t))\Big]=g(t, x(t)), \,\,\, \text{a.e.}
\,\,\, t\in J,\\
x(t_{0})=x_{0}\in \mathbb{R},
\end{array}
\right.
 \end{equation}
 where $f, g\in \mathcal{C}(J\times \mathbb{R}, \mathbb{R}).$

 In recent years,  a number of excellent results concerning  the existence of solutions and their approximation solutions
 for  nonlinear initial value problems of first and second orders hybrid functional, differential,
 integrodifferential equations have been obtained by   Dhage et al.
 in \cite{Dhag4}-\cite{Dhag7}. In \cite{Dhag4},
 the existence and approximation result for the following first order
nonlinear initial value problem of hybrid functional integrodifferential equations  is proven
\begin{equation}\label{3IE1}
 \left\{\begin{array}{ll}
\frac{d}{dx}\Big[x(t)-f(t,x(t))\Big]=\int_{0}^{t}g(s, x_{s})ds,
\,\,\, t\in J,\\
x_{0}=\phi,
\end{array}
\right.
 \end{equation}
 where $\phi, f, g\in C(J\times \mathbb{R}, \mathbb{R}).$
 In \cite{Dhag5}, it is proven some basic hybrid fixed point theorems for
 the sum and product of two operators defined in a Banach algebra.
 And applications of the newly developed abstract hybrid fixed point theorems
 to nonlinear hybrid linearly perturbed and quadratic integral equations
 are given. In \cite{Dhag6}, it is studied the following  nonlinear initial value
problem of second order hybrid functional differential equations
\begin{equation}\label{4IE1}
 \left\{\begin{array}{ll}
\frac{d}{dx}\Big[x'(t)-f(t,x_{t})\Big]=g(t, x_{t}),
\,\,\, t\in J,\\
x_{0}=\phi,\,\,\,x'(0)=\eta
\end{array}
\right.
 \end{equation}
 where $f, g \in C(J\times \mathbb{R}, \mathbb{R}).$
The existence and approximation result for these differential equations is proven.
In \cite{Dhag7}, it is  investigated  a system of two non-homogeneous
boundary value problems of coupled hybrid integro-differential
equations of fractional order.

In line with the above works, our purpose in this paper is to
developed  a new version of Kransoselskii's  fixed-point
theorem under a weak  contraction condition and utilizing this  theorem
to show the existence of a solution of the following  fractional hybrid differential equation involving
 the Riemann-Liouville differential and integral operators  orders of $0<\alpha<1$ and $\beta>0:$
\begin{equation}\label{IE1}
 \left\{\begin{array}{ll}
D^{\alpha}[x(t)-f(t, x(t))]=g(t, x(t), I^{\beta}(x(t))), \,\,\, \text{a.e.}
\,\,\, t\in J,\,\, \beta>0,\\
x(t_{0})=x_{0},
\end{array}
\right.
 \end{equation}
 where $J=[t_{0}, t_{0}+a],$ for some fixed $t_{0}\in \mathbb{R}$ and $a>0$
 and $f\in C(J\times \mathbb{R}, \mathbb{R}),$
 $g\in \mathcal{C}(J\times \mathbb{R}\times \mathbb{R}, \mathbb{R}).$
 The paper is organized as follows. In Section 2, we review some definitions
and rigorous results of fixed point theory, fractional calculus and
partially ordered Banach space. In Section 3, we prove
a new version of Kransoselskii's fixed-point
theorem under a weak  contraction condition. Utilizing this theorem, in Section 4, we
 prove the existence of a solution of a fractional hybrid differential equation involving
 the Riemann-Liouville differential and integral operators  orders of $0<\alpha<1$ and $\beta>0.$
In Section 5, we provide an illustrative example to highlight the realized improvements.

\section{Preliminaries}
\subsection{Integrals and derivatives of fractional order}
In this subsection, we recall some basic notions, concepts and definitions of integrals and derivatives of fractional order. The following are discussions of some of the concepts we will need.
Let $C(J\times \mathbb{R}\times \mathbb{R}, \mathbb{R})$ denote the class of continuous functions
$f:J\times \mathbb{R}\times \mathbb{R}\rightarrow \mathbb{R}$ and
$\mathcal{C}(J\times \mathbb{R}\times \mathbb{R}, \mathbb{R})$  denote the class of functions
$g:J\times \mathbb{R}\times \mathbb{R}\rightarrow \mathbb{R}$  such that
\begin{enumerate}
  \item [(i)] the map $t\rightarrow g(t, x, y)$ is measurable for each $x, y\in \mathbb{R},$
  \item [(ii)] the map $x\rightarrow g(t, x, y)$  is continuous for each $x\in \mathbb{R},$
  \item [(iii)] the map $y\rightarrow g(t, x, y)$  is continuous for each $y\in \mathbb{R}.$
\end{enumerate}
The class $\mathcal{C}(J\times \mathbb{R}\times \mathbb{R}, \mathbb{R})$ is called the Carath\'{e}odory class of functions on
$J\times \mathbb{R}\times \mathbb{R},$ which are Lebesgue integrable when bounded by a Lebesgue integrable function on $J.$

\begin{defn}[\cite{Kilbas}]
The form of the Riemann-Liouville fractional integral operator of order
$\alpha>0,$ of function $f\in L^{1}(\mathbb{R}^{+})$ is defined as
$$
I^{\alpha}f(x)=\frac{1}{\Gamma(\alpha)}\int_{0}^{x}(x-s)^{\alpha-1}f(s)ds.
$$
\end{defn}


\begin{defn}[\cite{Kilbas}]
The Riemann-Liouville derivative of fractional order $\alpha>0$ of a continuous function
$f:(0, \infty)\rightarrow \mathbb{R}$ is defined as
$$
D^{\alpha}f(x)=\frac{1}{\Gamma(m-\alpha)}\frac{d^{m}}{dx^{m}}\int_{0}^{x}(x-s)^{m-\alpha-1}f(s)ds,
$$
where $m=[\alpha]+1.$
\end{defn}
The following lemma will be used in the sequel.
\begin{lem}[\cite{Kilbas, Podlu}] Let $0<\alpha<1$ and $f\in L^{1}(0,1).$ Then
\begin{itemize}
  \item [(1)] the equality $D^{\alpha}I^{\alpha}f(x)=f(x)$ holds;
  \item [(2)] the equality
  $$
  I^{\alpha}D^{\alpha}f(x)=f(x)-\frac{[D^{\alpha-1}f(x)]_{x=0}}{\Gamma(\alpha)}x^{\alpha-1}
  $$
  holds for almost everywhere on $J.$
\end{itemize}
\end{lem}
\subsection{Partially ordered Banach space}
Throughout this paper  $X:=(X, \|\cdot\|)$ stands for a real Banach space. Let $P:=P_{X}$  be a  nonempty subset of $X$.
\begin{defn}
$P$ is called a \emph{closed convex cone} (or shortly, \emph{cone}) with vertex 0 if
\begin{itemize}
  \item [(1)] $P$ is closed, non-empty and $P\neq \{0\};$
  \item [(2)]  $\alpha x+\beta y \in P$ for all $x, y \in P$ and non-negative real numbers $\alpha, \beta;$
  \item [(3)] $P\cap (-P)=\{0\};$
  \item [(4)] a cone $P$ is called to be positive if
  $P\circ P\subseteq P,$ where $\circ$ is a multiplication composition in $X.$
\end{itemize}
\end{defn}
We define an order relation $\preceq$ in X as follows.
Let $x,y \in X.$ Then $x\preceq y$ if and only if $y-x\in P.$
The notation $y\prec x$ indicates that $y\preceq x$  and $x\neq y,$
while $y\ll x$ will show $x-y \in \text{int} P,$ where $\text{int} P$
denotes the interior of $P$. From now on, it is
assumed that $\text{int} P\neq \O.$
The triple $(X, \|\cdot\|, \preceq)$
is called a \textit{partially ordered Banach space}.
A cone $P$ is said to be \textit{normal} if the 
norm
$\|\cdot \|$ is semi-monotone increasing on $P,$
that is, there is a constant $N>0$ such that
$\|x\|\leq N \|y\|$ for all $x,y \in X$ with $x\preceq y.$
It is known that if the cone $P$ is normal in $X,$
then every order-bounded set in $X$ is norm-bounded.
The details of cones and their properties appear in Heikkil\"{a} and Lakshmikantham \cite{Hei}.
Next we define an \textit{upper comparable property}
and use in the proof of the main theorem.
\begin{defn}
We say that a partially ordered Banach space 
$(X, \|\cdot\|, \preceq)$ has an upper comparable property if for every $x,y\in X$
 there exists $z\in X$ such that 
 $x\preceq z$ and $y\preceq z.$
\end{defn}

\section{Main Results}
The aim of this section is to prove the following  a new version of Kransoselskii's fixed-point theorem
under a ($\psi, \theta, \varphi$)-weak contraction. We need the following definition.
\begin{defn}
The function $\psi:[0, \infty)\rightarrow [0, \infty)$
 is called a generalized altering distance function if the following properties are satisfied:
 \begin{itemize}
   \item [(1)] $\psi$ is a lower semi-continuous and non-decreasing;
   \item [(2)] $\psi(t)=0$ if and only if $t=0.$
 \end{itemize}
\end{defn}
Our main theorem is the following.
\begin{thm}\label{Main1}
Let $(X, \|\cdot\|, \preceq)$ be a partially ordered Banach space with an upper comparable property. Assume that $S$ be
a nonempty, closed, convex, and bounded subset of $X.$ Let $A:X\rightarrow X$ and
$B:S\rightarrow X$ be two operators verifying the following hypothesises:
\begin{itemize}
  \item [(a)] there exist a generalized  altering distance function $\psi:[0, \infty)\rightarrow [0, \infty),$ an upper semi-continuous function
  $\theta:[0, \infty)\rightarrow [0, \infty),$ and a lower semi-continuous function $\varphi:[0, \infty)\rightarrow [0, \infty)$
  such that
  $$
  \psi(\|Ax-Ay\|)\leq \theta(\|x-y\|)-\varphi(\|x-y\|), \,\,\,\,\, \text{for}\,\,\,\,\, x\succeq y,
  $$
  where $\theta(0)=\varphi(0)=0$ and $\psi(t)-\theta(t)+\varphi(t)>0$ for all $t>0;$
  \item [(b)] there exists $x_{0}\in X$ such that $x_{0}\preceq Ax_{0};$
      \item [(c)] $A$ is continuous, non-decreasing (w.r.t. $\preceq$) and $(I-A)^{-1}$ exists and it is continuous;
  \item [(d)] $B$ is continuous, $BS$ is contained  in a compact subset of positive cone $P;$
  \item [(e)] $x=Ax+By$ $\Rightarrow$ $x\in S$ for all $y\in S.$
\end{itemize}
Then the operator $A+B$ has a fixed point  in $S.$
\end{thm}
\begin{proof}
Let us fix arbitrarily $y\in S.$ Define $fx:=Ax+By.$ We show that $f$ has a fixed point in $S.$
By hypothesis $(b),$ there exists $x_{0}\in X$ such that $x_{0}\preceq Ax_{0}.$ Since $By\in P$ for all
$y\in S$ one has $0\preceq By.$ It implies $x_{0}\preceq Ax_{0}\preceq Ax_{0}+By=fx_{0}:=x_{1}.$ Because the operator
$A$ is non-decreasing we have $Ax_{0}\preceq Ax_{1}=A(Ax_{0}+By).$ It implies
$x_{1}=Ax_{0}+By\preceq Ax_{1}+By=A(Ax_{0}+By)+By=f^{2}x_{0}:=x_{2}.$ Proceeding by induction, we
obtain $x_{n+1}:=fx_{n}=Ax_{n}+By$ such that $x_{n}\preceq x_{n+1}$ for all $n\geq 0.$ If
$x_{n}=x_{n+1}$ for some  $n\geq 0$ then $f$ has a fixed point. Assume that
$x_{n}\neq x_{n+1}$ for all $n\geq 0.$ By using hypothesis $(a),$ we get
\begin{equation}\label{Eq1}
\begin{split}
\psi(\|x_{n}-x_{n+1}\|)& =\psi(\|fx_{n-1}-fx_{n}\|)\\
 & = \psi(\|Ax_{n-1}-Ax_{n}\|)\leq \theta(\|x_{n-1}-x_{n}\|)-\varphi(\|x_{n-1}-x_{n}\|).
\end{split}
\end{equation}
On the other hand we have $\psi(\|x_{n-1}-x_{n}\|)-\theta(\|x_{n-1}-x_{n}\|)+\varphi(\|x_{n-1}-x_{n}\|)>0$
since $\|x_{n-1}-x_{n}\|>0.$ It implies
$$
\frac{\psi(\|x_{n}-x_{n+1}\|)}{\psi(\|x_{n-1}-x_{n}\|)}\leq \frac{\theta(\|x_{n-1}-x_{n}\|)-\varphi(\|x_{n-1}-x_{n}\|)}{\psi(\|x_{n-1}-x_{n}\|)}<1.
$$
Thus
\begin{equation}\label{Eq2}
\psi(\|x_{n}-x_{n+1}\|)<\psi(\|x_{n-1}-x_{n}\|).
\end{equation}
Since $\psi$ is a generalized altering distance function we get
\begin{equation}\label{Eq3}
\|x_{n}-x_{n+1}\|<\|x_{n-1}-x_{n}\|.
\end{equation}
Hence the sequence $\big(\|x_{n}-x_{n+1}\|\big)_{n}$ is decreasing  and bounded below.
We show that 
\begin{equation}\label{20201}
\lim_{n\rightarrow \infty}\|x_{n}-x_{n+1}\|=0.
\end{equation}
 Suppose $\lim_{n\rightarrow \infty}\|x_{n}-x_{n+1}\|=r\neq 0.$ 
For $k\geq 1,$ define
\begin{align*}
a_{k} &= \sup\Big\{\theta(\|x_{n-1}-x_{n}\|)-\varphi(\|x_{n-1}-x_{n}\|):\,\,\,n\geq k\Big\};\\
b_{k} &= \inf\Big\{\psi(\|x_{n}-x_{n+1}\|):\,\,\,\, n\geq k\Big\}.
\end{align*}
It is clear
\begin{align*}
a_{n} & \geq\theta(\|x_{n-1}-x_{n}\|)-\varphi(\|x_{n-1}-x_{n}\|)\\
b_{n} & \leq\psi(\|x_{n}-x_{n+1}\|).
\end{align*}
These and from (\ref{Eq1}) it follows that
\begin{equation}\label{mayEq1}
b_{n}\leq \psi(\|x_{n}-x_{n+1}\|)\leq \theta(\|x_{n-1}-x_{n}\|)-\varphi(\|x_{n-1}-x_{n}\|)\leq a_{n}.
\end{equation}
By using the definitions of $\psi, \theta$ and $\varphi$ we get
\begin{equation}
\begin{aligned}[b]
\lim_{n\rightarrow \infty} a_{n} &=\limsup_{n\rightarrow \infty}\theta(\|x_{n-1}-x_{n}\|)-\varphi(\|x_{n-1}-x_{n}\|)\leq\theta (r)-\varphi(r),\\
\lim_{n\rightarrow \infty} b_{n} &=\liminf_{n\rightarrow \infty}\psi(\|x_{n}-x_{n+1}\|) \geq\psi(r).
\end{aligned}
\label{mayEq2}
\end{equation}
As a consequence of (\ref{mayEq1}) and (\ref{mayEq2}) we obtain
$$
\psi(r)\leq \theta (r)-\varphi(r)
$$
which is a contradiction to hypothesis $(a).$
Hence $r=0.$ Next we show that the sequence $(x_{n})$ is a Cauchy sequence in $X.$
Assume the contrary, that is, $(x_{n})$ is not a Cauchy sequence.
Then we can find an $\varepsilon>0$ and  sub-sequences $(x_{n_{i}})$ and
$(x_{m_{i}})$ of $(x_{n})$ with $n_{i}>m_{i}>i$ such that
\begin{equation}\label{Eq4}
\|x_{m_{i}}-x_{n_{i}}\|\geq\varepsilon.
\end{equation}
Let $n_{i}$ be the smallest index with $n_{i}>m_{i}>i$ and satisfying (\ref{Eq4}).
This means
\begin{equation}\label{Eq5}
\|x_{m_{i}}-x_{n_{i}-1}\|<\varepsilon.
\end{equation}
By using the triangular inequality we get
\begin{equation}\label{Eq6}
\|x_{m_{i}}-x_{n_{i}}\|\leq \|x_{m_{i}}-x_{n_{i}-1}\|+\|x_{n_{i}-1}-x_{n_{i}}\|.
\end{equation}
Combining (\ref{Eq4})-(\ref{Eq6}) we get
$$
\varepsilon \leq \|x_{m_{i}}-x_{n_{i}}\|<\|x_{n_{i}-1}-x_{n_{i}}\|+\varepsilon.
$$
Taking the limit as $i\rightarrow \infty$ we have
\begin{equation}\label{Eq7}
\lim_{i\rightarrow \infty}\|x_{m_{i}}-x_{n_{i}}\|=\varepsilon.
\end{equation}
On the other hand, the triangular inequality
implies
\begin{equation}\label{Eq8}
\|x_{m_{i}}-x_{n_{i}-1}\|\leq \|x_{m_{i}}-x_{n_{i}}\|+\|x_{n_{i}}-x_{n_{i}-1}\|.
\end{equation}
Letting $i\rightarrow \infty$ in (\ref{Eq6}) and (\ref{Eq8}) we obtain
$$
\varepsilon \leq \lim_{i\rightarrow \infty}\|x_{m_{i}}-x_{n_{i}-1}\|\leq\varepsilon.
$$
Hence
\begin{equation}\label{Eq9}
\lim_{i\rightarrow \infty}\|x_{m_{i}}-x_{n_{i}-1}\|=\varepsilon.
\end{equation}
Similarly, it can be shown that
\begin{equation}\label{Eq10}
\lim_{i\rightarrow \infty}\|x_{m_{i}+1}-x_{n_{i}}\|=\varepsilon.
\end{equation}
From hypothesis $(a)$ it follows that
\begin{equation}\label{Eq11}
\begin{split}
\psi(\|x_{m_{i}+1}-x_{n_{i}}\|)&=\psi(\|fx_{m_{i}}-fx_{n_{i}-1}\|)
=\psi(\|Ax_{m_{i}}-Ax_{n_{i}-1}\|)\\
&\leq \theta(\|x_{m_{i}}-x_{n_{i}-1}\|)-\varphi(\|x_{m_{i}}-x_{n_{i}-1}\|)
\end{split}
\end{equation}
since $x_{m_{i}}\preceq x_{n_{i}-1}.$
The same manner as in (\ref{mayEq1}) and (\ref{mayEq2})
it can be shown
 \begin{equation}
\begin{aligned}[b]
&\limsup_{i\rightarrow \infty}\theta(\|x_{m_{i}}-x_{n_{i}-1}\|)-\varphi(\|x_{m_{i}}-x_{n_{i}-1}\|)\leq\theta (\varepsilon)-\varphi(\varepsilon),\\
&\liminf_{i\rightarrow \infty}\psi(\|x_{m_{i}+1}-x_{n_{i}}\|)\geq
 \psi(\varepsilon),
\end{aligned}
\label{mayEq6}
\end{equation}
and $\psi(\varepsilon)\leq\theta (\varepsilon)-\varphi(\varepsilon)$ which is a contradiction due to $\varepsilon>0.$
Therefore $(x_{n})$ is a Cauchy sequence. Since $X$ is a Banach space
there exists $x^{*}\in X$ such that
$$
\lim_{n\rightarrow \infty}x_{n}=x^{*}.
$$
According to hypothesises $(c)$ and $(d)$ the operator $f$
is continuous, thus
$x^{*}=\lim_{n\rightarrow \infty}x_{n+1}=\lim_{n\rightarrow \infty}fx_{n}=fx^{*}.$
Hence $x^{*}\in X$ is a fixed point of $f.$
Next we show that $f$ has a unique fixed point.
Suppose $f$ has another fixed point $y^{*}\in X.$ Since $X$ has an upper comparable property
there exists $z\in X$ such that $x^{*}\preceq z$ and $y^{*}\preceq z.$ First, we assume 
$x^{*}\prec z$ and $y^{*}\prec z.$ In this case, because the operator $A$ is non-decreasing we have $Ax^{*}\preceq Az$ it implies $Ax^{*}+By\preceq Az+By$ that is $fx^{*}\preceq fz.$ Since $x^{*}$ is a fixed point of $f$ we have $x^{*}\preceq fz.$ Proceeding by induction, we obtain $x^{*}\preceq z_{n}$  for all $n\geq 1$ where $z_{n}:=f^{n}z.$ For definiteness we assume
$x^{*}\neq z_{n}$ for all $n\geq 1.$
By using hypothesis $(a)$ we get 
\begin{equation}\label{Eq120}
\begin{split}
\psi(\|x^{*}-z_{n+1}\|)& =\psi(\|fx^{*}-fz_{n}\|)\\
 & = \psi(\|Ax^{*}-Az_{n}\|)\leq \theta(\|x^{*}-z_{n}\|)-\varphi(\|x^{*}-z_{n}\|).
\end{split}
\end{equation}
On the other hand we have $\psi(\|x^{*}-z_{n}\|)-\theta(\|x^{*}-z_{n}\|)+\varphi(\|x^{*}-z_{n}\|)>0$
since $\|x^{*}-z_{n}\|>0.$ It implies
$$
\frac{\psi(\|x^{*}-z_{n+1}\|)}{\psi(\|x^{*}-z_{n}\|)}\leq \frac{\theta(\|x^{*}-z_{n}\|)-\varphi(\|x^{*}-z_{n}\|)}{\psi(\|x^{*}-z_{n}\|)}<1.
$$
Thus
\begin{equation}\label{Eq220}
\psi(\|x^{*}-z_{n+1}\|)<\psi(\|x^{*}-z_{n}\|).
\end{equation}
Since $\psi$ is a generalized altering distance function we get
\begin{equation}\label{Eq320}
\|x^{*}-z_{n+1}\|<\|x^{*}-z_{n}\|.
\end{equation}
Hence the sequence $\big(\|x^{*}-z_{n}\|\big)_{n}$ is decreasing  and bounded below.
The same manner as in the proof of relation 
(\ref{20201}) it can be shown that 
$$
\lim_{n\rightarrow \infty}\|x^{*}-z_{n}\|=0.
$$
Similarly, we can deduce that 
$$
\lim_{n\rightarrow \infty}\|y^{*}-z_{n}\|=0.
$$
On the other hand we have 
$$
\|x^{*}-y^{*}\|\leq \|x^{*}-z_{n}\|+\|y^{*}-z_{n}\|.
$$
Taking the limit as $n\rightarrow \infty$
we get $x^{*}=y^{*}.$ 
Now we assume $x^{*}=z$ and $y^{*}\prec z$ that is $y^{*}\prec x^{*}.$ It implies $\|x^{*}-y^{*}\|>0.$
By using hypothesis $(a)$ we obtain
 \begin{equation}\label{Eq1201}
\begin{split}
\psi(\|x^{*}-y^{*}\|)& =\psi(\|fx^{*}-fy^{*}\|)\\
 & = \psi(\|Ax^{*}-Ay^{*}\|)\leq \theta(\|x^{*}-y^{*}\|)-\varphi(\|x^{*}-y^{*}\|).
\end{split}
\end{equation}
On the other hand we have $\psi(\|x^{*}-y^{*}\|)-\theta(\|x^{*}-y^{*}\|)+\varphi(\|x^{*}-y^{*}\|)>0$
since $\|x^{*}-y^{*}\|>0.$ 
But this contradicts the inequality (\ref{Eq1201}). Hence $x^{*}=y^{*}.$ The case $x^{*}\prec z$ and $y^{*}=z$ is similar to the above case and the case 
$x^{*}=z$ and $y^{*}=z$ is trivial.  
So, in all cases, we have shown that 
$x^{*}=y^{*}.$ 
Hence $f$ has a unique fixed point, that is, there exists a unique $x^{*}\in X$  such that 
\begin{equation}\label{Eq12}
x^{*}=Ax^{*}+By.
\end{equation}
Hypothesis $(e)$ implies that $x^{*}\in S.$
From the equality (\ref{Eq12}) it follows that  $(I-A)x^{*}=By$ for all $y\in S.$
Since the operator $(I-A)^{-1}$ exists and continuous we have
$x^{*}=(I-A)^{-1}By\in S$ for all $y\in S.$ Now according to $(d)$
the set $BS$ is contained in a compact subset of $P$, while $(I-A)^{-1}$ is continuous, and so $(I-A)^{-1}BS$
is contained in a compact subset of $P.$ (For the proof of this in general metric spaces, see \cite{Kr} pp. 412).
From hypothesis $(e)$ and equality  (\ref{Eq12})
it follows that  $(I-A)^{-1}BS$ is contained in a compact subset of the closed set $S.$
By Schauder's second theorem (see \cite{Smart} pp. 25) the operator $(I-A)^{-1}B$ has a fixed point in $S,$
that is, there exists $z^{*}\in S$ such that $z^{*}=(I-A)^{-1}Bz^{*}.$ This implies
$Az^{*}+Bz^{*}=z^{*}.$ Theorem \ref{Main1} is proved.
\end{proof}

\begin{rem} Note that
\begin{itemize}
\item  the main idea of the proof of the existence of a fixed point of the mapping
$f$ has been borrowed from \cite{Habib};
\item in the case of $\psi$-altering distance function, the $(\psi,\theta,\varphi)$-weak
contraction condition has been  successfully applied in multidimensional fixed point theorems and their applications to the system of matrices equations and nonlinear integral equations (see, for instance \cite{akhad1}-\cite{akhad3});
\item the generalized $(\psi,\theta,\varphi)$-weak
contraction condition extends the notion of $D$-contraction condition which is defined by a
dominating function or, in short, $D$-function (see, for instance \cite{Dhag1} and \cite{Dhag6, Dhag7}).
\item  Theorem \ref{Main1} extends the main theorem of \cite{Dhag0}.
\end{itemize}
\end{rem}

\section{Applications of Theorem \ref{Main1}}
In this section, by applying Theorem \ref{Main1} we study the existence of
a solution of the fractional hybrid differential equation (\ref{IE1})
under the following assumptions.

\subsubsection{Hypothesis for FHDE (\ref{IE1})}
We assume:
\begin{itemize}
\item[(H1)] The function $F_{t}:\mathbb{R}\rightarrow \mathbb{R}$ defined as
$F_{t}(x):=x-f(t,x)$ is strictly increasing in $\mathbb{R}$ for all $t\in J.$
\item[(H2)] The function $f(t, \cdot)$ satisfies  the following weak contraction condition
$$
|f(t,x)-f(t,y)|\leq \arctan\big(|x-y|\big)
$$
for all $t\in J$ and $x,y\in \mathbb{R}.$
\item[(H3)] The function $f(t, 0)$ satisfies
$$
f(t,0)-f(t_{0}, x_{0})+x_{0}\geq 0
$$
for all $t\in J.$
\item[(H4)] The function $g$ is non-negative and there exists a continuous function $h\in C(J, \mathbb{R})$ such that
$$
0\leq g(t,x,y)\leq h(t)
$$
for almost all $t\in J$ and for all $x,y\in \mathbb{R}.$
\end{itemize}
The following lemma is useful in what follows.

\begin{lem}[\label{Lu1}\cite{Lu}]
Assume that hypothesis (H1) holds. Then, for any $y\in C(J, \mathbb{R})$
and $\alpha\in (0,1)$ the function $x\in C(J, \mathbb{R})$
is a solution of  FHDE
\begin{equation}\label{4E2}
D^{\alpha}[x(t)-f(t, x(t))]=y(t), \,\, t\in J,
\end{equation}
with the initial condition $x(t_{0})=x_{0}$
if and only if $x(t)$ satisfies the hybrid integral equation (in short, HIE)
\begin{equation}\label{4E3}
x(t)=x_{0}-f(t_{0}, x_{0})+ f(t, x(t))
+\frac{1}{\Gamma(\alpha)}\int_{t_{0}}^{t}(t-s)^{\alpha-1}y(s)ds,\,\,\,t\in J.
\end{equation}
\end{lem}
Now, we are in a position to prove the following existence theorem for FHDE (\ref{IE1}).
\begin{thm}\label{Main2}
Assume that hypotheses (H1)-(H4) hold. Then
FHDE (\ref{IE1}) has a solution  in $C(J, \mathbb{R}).$
\end{thm}
\begin{proof}
Set $X=C(J, \mathbb{R}).$ Let $\|\cdot\|$ be the maximum norm
in $X,$ that is, $\|x\|=\max_{t\in J}|x(t)|.$ Clearly,
$X$ is a Banach space with respect to this norm. We claim that the space $X$ has an upper 
comparable property. Indeed, for any $x(t), y(t)\in X$ we can find $z(t)\in X$ as 
$z(t)=\max_{t\in J}\{x(t), y(t)\}.$
It is clear that $x(t)\leq z(t)$ and $y(t)\leq z(t).$ Define a subset $S$  of $X$ as follows.
$$
S=\Big\{x\in X:\,\,\, \|x\|\leq M\Big\}
$$
where $M=|x_{0}-f(t_{0}, x_{0})|+\frac{\pi}{2}+L+\frac{|J|^{\alpha}}{\Gamma(\alpha+1)}\|h\|$
and $L=\max_{t\in J} f(t, 0).$ Clearly, $S$ is a closed, convex and bounded subset of the Banach space
$X.$
Define $P=\big\{x\in X:\,\,\, x\geq 0\big\}.$
It is obvious that the set $P$ is a positive cone and normal in $X.$
For a given $y\in S,$ consider the following generalized fractional hybrid differential equation
involving  the Riemann-Liouville  differential and integral operators  orders of $0<\alpha<1$ and $\beta>0:$
\begin{equation}\label{4E1}
 \left\{\begin{array}{ll}
D^{\alpha}[x(t)-f(t, x(t))]=g(t, y(t), I^{\beta}(y(t))), \,\,\, \text{a.e.}
\,\,\, t\in J,\,\, \beta>0,\\
x(t_{0})=x_{0}
\end{array}
\right.
 \end{equation}
 where $J=[t_{0}, t_{0}+a],$ for some fixed $t_{0}, a \in \mathbb{R}^{+}$
 and $f\in C(J\times \mathbb{R}, \mathbb{R}),$
 $g\in \mathcal{C}(J\times \mathbb{R}\times \mathbb{R}, \mathbb{R}).$
 Assume $f$ and $g$ satisfy the assumptions (H1)-(H4). Then by Lemma \ref{Lu1}
  the equation (\ref{4E1}) is equivalent to the nonlinear HIE
  \begin{equation}\label{4E4}
  x(t)=x_{0}-f(t_{0}, x_{0})+f(t, x(t))+\frac{1}{\Gamma(\alpha)}\int_{t_{0}}^{t}(t-s)^{\alpha-1}
   g(s, y(s), I^{\beta}(y(s)))ds.
   \end{equation}
Define  operators $A:X\rightarrow X$ and $B:S\rightarrow X$ by
\begin{equation}\label{4E5}
Ax(t)=x_{0}-f(t_{0}, x_{0})+ f(t, x(t)),
\,\,\,\,t\in J,
\end{equation}
and
\begin{equation}\label{4E6}
By(t)=\frac{1}{\Gamma(\alpha)}\int_{t_{0}}^{t}(t-s)^{\alpha-1}
   g(s, y(s), I^{\beta}(y(s)))ds,
   \,\,\,\,t\in J.
\end{equation}
Then HIE (\ref{4E4}) is transformed into the operator equation as
\begin{equation}\label{4E7}
x(t)=Ax(t)+By(t),
\,\,\,\,t\in J.
\end{equation}
We will show that the operators $A$ and $B$ satisfy all hypothesises of Theorem \ref{Main1}.
First, we show that the operator  $A$ satisfies hypothesis $(a)$ of Theorem \ref{Main1} with
$$
\psi(t)=t,\,\,\,\theta(t)=\arctan(t)
\,\,\, \text{and}\,\,\, \varphi(t)=0.
$$
Let $x,y\in X.$ Then by  hypothesis (H2) we have
$$
|Ax(t)-Ay(t)|=|f(t, x(t))-f(t, y(t))|\leq \arctan(|x(t)-y(t)|)
\leq \arctan(\|x-y\|).
$$
Taking maximum over $t,$
we obtain
$$
\|Ax-Ay\|\leq \arctan(\|x-y\|).
$$
One can easily see that $\psi(t)-\theta(t)+\varphi(t)=t-\arctan(t)>0$
for all $t>0$ and $\theta(0)=\varphi(0)=0.$ Hence the operator  $A$ satisfies hypothesis $(a).$
Next, we show that $A$ satisfies hypothesis $(b).$
Let $x_{0}(t)\equiv 0.$ Then  by  (H3) we have
$$
Ax_{0}(t)=x_{0}-f(t_{0}, x_{0})+f(t, 0)\geq 0= x_{0}(t).
$$
Now we show that $(I-A)^{-1}$ exists and continuous.
By  hypotheses (H1) and (H2) the function
$F_{t}(x)=x-f(t,x)$ is strictly increasing and continuous
for all $t\in J.$ Therefore
$F^{-1}_{t}$ exists and continuous for all $t\in J.$
 Consider the operator $T:X\rightarrow X$
defined as $Tx(t)=F^{-1}_{t}(x(t)).$
One can easily see that
$$
(I-A)\circ T=F_{t}\circ F^{-1}_{t}=I
\,\,\,\,\text{and}\,\,\,\,
T\circ (I-A) =F^{-1}_{t}\circ F_{t}=I.
$$
Hence $T$ is the inverse of $I-A$ and continuous
since $F^{-1}_{t}$ is continuous.
Next, we show that the operator $B$ satisfies
hypothesis $(d)$ of Theorem \ref{Main1}.
By (H4) the function $g(\cdot, x,y)$ is almost everywhere non-negative  thus
for any $y\in S$ we get
$$
By(t)=\frac{1}{\Gamma(\alpha)}\int_{t_{0}}^{t}(t-s)^{\alpha-1}
   g(s, y(s), I^{\beta}(y(s)))ds\geq 0.
$$
Hence $BS$ is a subset of the positive cone $P.$
Now, we show that $B$ is continuous on $S.$
Let $(y_{n})$ be a sequence in $S$ converging to a point $y\in S.$
Then, by Lebesgue dominated convergence theorem we have
\begin{equation}\nonumber
\begin{split}
\lim_{n\rightarrow\infty}By_{n} & =
\lim_{n\rightarrow\infty} \frac{1}{\Gamma(\alpha)}\int_{t_{0}}^{t}(t-s)^{\alpha-1}
   g(s, y_{n}(s), I^{\beta}(y_{n}(s)))ds\\
 & = \frac{1}{\Gamma(\alpha)}\int_{t_{0}}^{t}(t-s)^{\alpha-1}
   \lim_{n\rightarrow\infty}g(s, y_{n}(s), I^{\beta}(y_{n}(s)))ds\\
 &=\frac{1}{\Gamma(\alpha)}\int_{t_{0}}^{t}(t-s)^{\alpha-1}
   g(s, y(s), I^{\beta}(y(s)))ds=By(t)
\end{split}
\end{equation}
for all $t\in J.$ Hence the operator $B$ is continuous.
Now, we show that $B$ is a compact operator on $S.$ It is enough
to show that $BS$ is a uniformly bounded and equicontinuous  set in $S.$
By hypothesis (H4) we have
\begin{equation}\nonumber
\begin{split}
|By(t)|&=\Big|\frac{1}{\Gamma(\alpha)}\int_{t_{0}}^{t}(t-s)^{\alpha-1}
   g(s, y(s), I^{\beta}(y(s)))ds\Big|\\
   & \leq\frac{1}{\Gamma(\alpha)}\int_{t_{0}}^{t}(t-s)^{\alpha-1}
   \Big|g(s, y(s), I^{\beta}(y(s)))\Big|ds\\
   &\leq \frac{1}{\Gamma(\alpha)}\int_{t_{0}}^{t}(t-s)^{\alpha-1}
   h(s)ds \leq \frac{|J|^{\alpha}}{\Gamma(\alpha+1)}\|h\|
   \end{split}
\end{equation}
for  $t\in J.$ Taking maximum over $t$ we get
$$
\|By\|\leq \frac{|J|^{\alpha}}{\Gamma(\alpha+1)}\|h\|.
$$
Thus the operator $B$ is uniformly bounded on $S.$
Let $t_{1}, t_{2}\in J$ with $t_{1}<t_{2}.$ Then, for any $x\in S,$
one has
\begin{equation}\nonumber
\begin{split}
|Bx(t_{1})&-Bx(t_{2})|=\\
&\Big|\frac{1}{\Gamma(\alpha)}\int_{t_{0}}^{t_{1}}(t_{1}-s)^{\alpha-1}
   g(s, y(s), I^{\beta}(y(s)))ds-
   \frac{1}{\Gamma(\alpha)}\int_{t_{0}}^{t_{2}}(t_{2}-s)^{\alpha-1}
   g(s, y(s), I^{\beta}(y(s)))ds\Big|\\
&\leq\Big|\frac{1}{\Gamma(\alpha)}\int_{t_{0}}^{t_{1}}(t_{1}-s)^{\alpha-1}
   g(s, y(s), I^{\beta}(y(s)))ds-
   \frac{1}{\Gamma(\alpha)}\int_{t_{0}}^{t_{1}}(t_{2}-s)^{\alpha-1}
   g(s, y(s), I^{\beta}(y(s)))ds\Big|\\
&+\Big|\frac{1}{\Gamma(\alpha)}\int_{t_{0}}^{t_{1}}(t_{2}-s)^{\alpha-1}
   g(s, y(s), I^{\beta}(y(s)))ds-
   \frac{1}{\Gamma(\alpha)}\int_{t_{0}}^{t_{2}}(t_{2}-s)^{\alpha-1}
   g(s, y(s), I^{\beta}(y(s)))ds\Big|\\
   & \leq \frac{\|h\|}{\Gamma(\alpha+1)}\Big[|(t_{2}-t_{0})^{\alpha}-(t_{1}-t_{0})^{\alpha}
   +(t_{2}-t_{1})^{\alpha}|+(t_{2}-t_{1})^{\alpha}\Big].
   \end{split}
\end{equation}
Hence, for  any $\varepsilon,$ there exists a $\delta>0$ such that
$$
|t_{1}-t_{2}|\leq \delta\,\,\,\,\,\Rightarrow
\,\,\,\,\, |Bx(t_{1})-Bx(t_{2})|\leq \varepsilon,
$$
for  $t_{1}, t_{2}\in J$ and for all $x\in S.$
This shows that $BS$ is  an equicontinuous set in $X.$
Hence according to  Arzela-Ascoli theorem  the set $BS$ is compact .
Finally, we show that hypothesis  $(e)$ of Theorem \ref{Main1} is satisfied.
Let $x\in X$ and $y\in S$ satisfy the equation $x=Ax+By.$ Then, by hypothesis (H2),
we have
\begin{equation}\nonumber
\begin{split}
|Ax(t)-By(t)|&\leq |Ax(t)|+|By(t)|\\
&\leq |x_{0}-f(t_{0}, x_{0})| +|f(t, x(t))|+
\Big|\frac{1}{\Gamma(\alpha)}\int_{t_{0}}^{t}(t-s)^{\alpha-1}
   g(s, y(s), I^{\beta}(y(s)))ds\Big|\\
&\leq |x_{0}-f(t_{0}, x_{0})| +|f(t, x(t))-f(t, 0)|+|f(t, 0)|\\
&+\frac{1}{\Gamma(\alpha)}\int_{t_{0}}^{t}(t-s)^{\alpha-1}
   |g(s, y(s), I^{\beta}(y(s)))|ds\\
&\leq |x_{0}-f(t_{0}, x_{0})| +
\arctan(\|x\|)+L+
\frac{|J|^{\alpha}}{\Gamma(\alpha+1)}\|h\|\\
&\leq |x_{0}-f(t_{0}, x_{0})| +\frac{\pi}{2}+L+
\frac{|J|^{\alpha}}{\Gamma(\alpha+1)}\|h\|.
 \end{split}
\end{equation}
Taking maximum over $t,$ we have
$$
\|x\|=\max_{t\in J}|Ax(t)-By(t)|\leq |x_{0}-f(t_{0}, x_{0})| +\frac{\pi}{2}+L+
\frac{|J|^{\alpha}}{\Gamma(\alpha+1)}\|h\|=M.
$$
Hence $x\in S.$ Thus all hypothesises  of Theorem \ref{Main1} are satisfied and so
the operator $A+B$ has a fixed point in $S,$ that is, there exists $z^{*}\in S$
such that $Az^{*}+Bz^{*}=z^{*}.$ As a result, FHDE (\ref{IE1}) has a solution in $S.$
This completes the proof of Theorem \ref{Main2}.

\section{Illustrative example}
Let $J=[0, 1].$ Denote by $X$  the set of continuous and non-negative functions
$f:J\rightarrow [0, \infty).$  In $X$  consider the following fractional hybrid differential equation:
 \begin{equation}\label{Example1}
 \left\{\begin{array}{ll}
D^{\frac{1}{2}}\big[x(t)-\tanh(t)\arctan\big(x(t)+1\big)\big]=t^{2}e^{t}|\sin\big(x(t)\big)|\frac{I^{\beta}\big(x(t)\big)}{1+I^{\beta}\big(x(t)\big)},
\\
x(0)=0,
\end{array}
\right.
 \end{equation}
where $t\in J$ and  $\beta>0.$
Observe that this equation is a special case of the FHDE (\ref{IE1}) if we set
$$
f(t,x(t))=\tanh(t)\arctan\big(x(t)+1\big)
$$
and
$$
g(t, x(t),I^{\beta}(x(t)))=t^{2}e^{t}|\sin\big(x(t)\big)|\frac{I^{\beta}\big(x(t)\big)}{1+I^{\beta}\big(x(t)\big)}.
$$
We show that the equation (\ref{Example1}) satisfies the hypothesis (H1)-(H4). We claim that the function
$$
F_{t}(x):=x-f(t,x)=x-\tanh(t)\arctan(x+1)
$$
is strictly increasing in $\mathbb{R}^{+}\cup\{0\}$ for all $t\in J.$  In order to show that
$F_{t}$ is strictly increasing it is sufficient to show that the partial derivative $\partial F_{t}/\partial x$
is positive  for all $t\in J.$ Indeed, since $\tanh(t)<1$ for $t\in J$ we have
$$
\frac{\partial F_{t}}{\partial x}=1-\frac{\tanh(t)}{1+(x+1)^{2}}>0.
$$
Hence, the hypothesis (H1) is satisfied.
Next we show that $f$ satisfies the hypothesis (H2). Let $x\in \mathbb{R}^{+}\cup\{0\}$
and $\delta>0.$  One can see
\begin{equation}\label{5Eq1}
\begin{split}
|f(t, x+\delta)-f(t, x)|&=\tanh(t)\Big|\arctan(x+\delta+1)-\arctan(x+1)\Big|\\
&\leq\tanh(t)\sup_{x\in \mathbb{R}^{+}\cup\{0\}}\Big|\arctan(x+\delta+1)-\arctan(x+1)\Big|.
 \end{split}
\end{equation}
We estimate $g_{\delta}(x):=\arctan(x+\delta+1)-\arctan(x+1).$
Since the arctangent function  is increasing by the arctangent subtraction formula we have
\begin{equation}\label{5Eq2}
\begin{split}
|g_{\delta}(x)|&=|\arctan(x+\delta+1)-\arctan(x+1)|\\
&=\Big|\arctan\Big(\frac{\delta}{1+(x+\delta+1)(x+1)}\Big)\Big|
\leq \arctan(\delta).
\end{split}
\end{equation}
For any $x,y\in \mathbb{R}^{+}\cup\{0\}$ with $x<y$
by setting $\delta:=y-x$ and combining the relations (\ref{5Eq1}) and (\ref{5Eq2}) we get
\begin{equation}\label{5Eq3}
\begin{split}
|f(t, x)-f(t, y)|&=|f(t, x+\delta)-f(t, x)|\\
&=\tanh(t)\Big|\arctan(x+\delta+1)-\arctan(x+1)\Big|\\
&\leq\tanh(t)\sup_{x\in \mathbb{R}^{+}\cup\{0\}}\Big|\arctan(x+\delta+1)-\arctan(x+1)\Big|\\
&\leq \tanh(t)\arctan(\delta)\leq \arctan(|x-y|)
 \end{split}
\end{equation}
since $\tanh(t)<1$ for $t\in J.$ Hence, $f$ satisfies the hypothesis (H2). It is easy to check
 hypothesis (H3) because
$$
f(t,0)-f(0,0)=\tanh(t)\arctan(1)-\tanh(0)\arctan(1)=\frac{\pi}{4}\tanh(t)\geq 0.
$$
Finally, we show that $g$ satisfies  hypothesis (H4) with $h(t)=t^{2}e^{t}.$
It is easy to see that $I^{\beta}\big(x(t)\big)\geq 0$ since $x(t)\geq 0.$ This implies
$$
g(t, x(t),I^{\beta}(x(t)))=t^{2}e^{t}|\sin\big(x(t)\big)|\frac{I^{\beta}\big(x(t)\big)}{1+I^{\beta}\big(x(t)\big)} \geq 0.
$$
On the other hand we have
$$
|\sin\big(x(t)\big)|\leq 1\,\,\,\, \text{and}\,\,\,\, \frac{I^{\beta}\big(x(t)\big)}{1+I^{\beta}\big(x(t)\big)}\leq 1
$$
which implies that
$$
g(t, x(t),I^{\beta}(x(t)))=t^{2}e^{t}|\sin\big(x(t)\big)|\frac{I^{\beta}\big(x(t)\big)}{1+I^{\beta}\big(x(t)\big)}
\leq t^{2}e^{t}.
$$
Thus
$$
0\leq g(t, x(t),I^{\beta}(x(t)))\leq t^{2}e^{t}.
$$
So, $g$ satisfies hypothesis (H4). Hence all (H1)-(H4) hypotheses 
are satisfied. Thus by Theorem \ref{Main2} we conclude that the hybrid differential equation
(\ref{Example1}) has a solution.
\end{proof}

\section*{Authors' contributions} All authors contributed equally and significantly in writing this article. All authors read and approve the final manuscript.

\section*{Competing interests} 
The authors declare that they have no competing interests.

\section*{Acknowledgement}
The authors are grateful to the editor-in-chief and referees for their accurate reading and useful
suggestions. We would like to thank  the
 Ministry of  Education of Malaysia for
 providing us with the Fundamental Research Grant Scheme 
 (FRGS/1/2018/STG06/UUM/02/13 Code  S/O 14192).

\bibliographystyle{amsplain,latexsym}

\end{document}